\documentclass[12pt,reqno,oneside,a4paper]{amsart}
\usepackage[british]{babel}

\usepackage{graphicx}
\usepackage{amscd}
\usepackage{amsmath}
\usepackage{amsfonts}
\usepackage{amssymb}
\usepackage{comment}
\usepackage{color}
\usepackage[usenames,dvipsnames,table,xcdraw]{xcolor}
\usepackage{pdfsync}
\usepackage{bbm}
\usepackage{dsfont}
\usepackage[colorlinks=true]{hyperref}
\usepackage{enumerate}
\usepackage{booktabs}
\usepackage{float}
\usepackage{wrapfig}
\usepackage{caption}
\usepackage{subcaption}
\usepackage{longtable}
\usepackage{listings}
\usepackage[T1]{fontenc}
\usepackage[scaled]{beramono}
\usepackage{tikz}
\usepackage{pgffor}
\usepackage[all]{xy}

\definecolor{trp}{rgb}{1,1,1}

\definecolor{red}{rgb}{1,0,.2}

\newtheorem{theorem}{Theorem}[section]
\theoremstyle{plain}

\newtheorem{claim}[theorem]{Claim}

\newtheorem{definition}[theorem]{Definition}

\newtheorem{lemma}[theorem]{Lemma}

\newtheorem{prop}[theorem]{Proposition}
\newtheorem{remark}[theorem]{Remark}

\numberwithin{equation}{section}

\newcommand{\R}{\mathbb{R}}

\newcommand{\ii}{\mathbf{i}}
\newcommand{\jj}{\mathbf{j}}

\newcommand{\iih}{\boldsymbol{\hat\imath}}

\newcommand{\ih}{\hat\imath}
\newcommand{\jh}{\hat\jmath}
\newcommand{\iiv}{\overline{\imath}}

\usepackage{anysize}

\usepackage{caption}

\definecolor{blue}{rgb}{0,0,1}

\definecolor{red}{rgb}{1,0,.7}

\begin{document}
\title[Intermediate dimensions of Bedford--McMullen carpets]{On the intermediate dimensions of \\ Bedford--McMullen carpets}

\author{Istv\'an Kolossv\'ary}
\address{Istv\'an Kolossv\'ary, \newline University of St Andrews,  School of Mathematics and Statistics, \newline St Andrews, KY16 9SS, Scotland} \email{itk1@st-andrews.ac.uk}

\thanks{2020 {\em Mathematics Subject Classification.} Primary 28A80 Secondary 28A78 37C45
\\ \indent
{\em Key words and phrases.} intermediate dimensions, Bedford--McMullen carpet, Hausdorff dimension, box dimension}

\begin{abstract}
The intermediate dimensions of a set $\Lambda$, elsewhere denoted by $\dim_{\theta}\Lambda$, interpolates between its Hausdorff and box dimensions using the parameter $\theta\in[0,1]$. Determining a precise formula for $\dim_{\theta}\Lambda$ is particularly challenging when $\Lambda$ is a Bedford--McMullen carpet with distinct Hausdorff and box dimension. In this direction, answering a question of Fraser, we show that $\dim_{\theta}\Lambda$ is strictly less than the box dimension of $\Lambda$ for every $\theta<1$, moreover, the derivative of the upper bound is strictly positive at $\theta=1$. We also improve on the lower bound obtained by Falconer, Fraser and Kempton.  
\end{abstract}

\maketitle

\thispagestyle{empty}
\section{Introduction and main results}\label{sec:01}

In fractal geometry, perhaps the most studied notions of dimension of a subset $F$ of $\R^d$ are its Hausdorff and box dimensions. Both quantities can be formulated by means of covers of the set $F$. A finite or countable collection of sets $\{U_i\}$ is a \emph{cover} of $F$ if $F\subseteq \bigcup_i U_i.$ Throughout, the diameter of a set $F$ is denoted by $|F|$.

The Hausdorff dimension of $F$ is
\begin{equation*}
\dim_{\mathrm H}F = \inf \big\{s\geq 0:\, \text{for all } \varepsilon>0 \text{ there exists a cover } \{U_i\} \text{ of } F \text{ such that } {\textstyle\sum}|U_i|^s\leq \varepsilon\big\},
\end{equation*}
see \cite[Section 3.2]{FalconerBook}, while the (lower) box dimension is
\begin{multline*}
\underline{\dim}_{\mathrm B}F = \inf \big\{s\geq 0:\, \text{for all } \varepsilon>0 \text{ there exists a cover } \{U_i\} \text{ of } F \text{ such that } \\ |U_i|=|U_j| \text{ for all } i,j \text{ and }  {\textstyle\sum} |U_i|^s\leq \varepsilon\big\},
\end{multline*}
see \cite[Chapter 2]{FalconerBook}. We commonly refer to the quantity ${\textstyle\sum} |U_i|^s$ as the \emph{cost} of the cover.

In other words, the Hausdorff dimension is the smallest possible exponent $s$ such that we can find an \emph{optimal covering strategy} of $F$ in the sense that the cost of these covers can be made arbitrarily small with \textbf{no} restrictions on the diameters of the covering sets. On the other hand, the box dimension gives the exponent when we restrict to coverings with sets of the \textbf{same} diameter.

In particular, if $\dim_{\mathrm H}F=\underline{\dim}_{\mathrm B}F$, then $F$ has an optimal covering strategy where each covering contains sets with equal diameter. However, if  $\dim_{\mathrm H}F<\underline{\dim}_{\mathrm B}F$, then it is natural to ask what different diameters are used in an optimal covering strategy for $\dim_{\mathrm H}F$? The discussion above suggests a way to interpolate between $\dim_{\mathrm H}F$ and $\underline{\dim}_{\mathrm B}F$.

Falconer, Fraser and Kempton~\cite{FFK2019} introduced a continuum of \emph{intermediate dimensions} that achieve this interpolation by imposing increasing restrictions on the relative sizes of covering sets governed by a parameter $0\leq \theta\leq 1$. The Hausdorff and box dimension are the two extreme cases when $\theta=0$ and $1$, respectively.

\begin{definition}\label{def:01}
For $0\leq \theta\leq 1$, the \emph{lower $\theta$-intermediate dimension} of a bounded set $F\subset \R^d$ is defined by
\begin{multline*}
\underline{\dim}_{\theta}F = \inf \{s\geq 0:\, \text{for all } \varepsilon>0 \text{ and all } \delta_0>0, \text{ there exists } 0<\delta\leq \delta_0 \\
\text{ and a cover } \{U_i\} \text{ of } F \text{ such that } \delta^{1/\theta}\leq |U_i| \leq \delta \text{ and } {\textstyle\sum} |U_i|^s\leq \varepsilon\},
\end{multline*}
while its \emph{upper $\theta$-intermediate dimension} is given by
\begin{multline}\label{eq:100}
\overline{\dim}_{\theta}F = \inf \{s\geq 0:\, \text{for all } \varepsilon>0 \text{ there exists } \delta_0>0 \text{ such that for all } 0<\delta\leq \delta_0, \\
\text{ there is a cover } \{U_i\} \text{ of } F \text{ such that } \delta^{1/\theta}\leq |U_i| \leq \delta \text{ and }  {\textstyle\sum} |U_i|^s\leq \varepsilon\}.
\end{multline}
For a given $\theta$, if the values of $\underline{\dim}_{\theta}F$ and $\overline{\dim}_{\theta}F$ coincide, then the common value is called the \emph{$\theta$-intermediate dimension} and is denoted by  $\dim_{\theta}F$.
\end{definition}
Thus, the restriction is to only consider covering sets with diameter in the range $[\delta^{1/\theta},\delta]$. As $\theta\to 0$, the $\theta$-intermediate dimension gives more insight into which scales are used in the optimal cover to reach the Hausdorff dimension. Intermediate dimensions can also be formulated using capacity theoretic methods and may be used to relate the box dimensions of the projections of a set to the Hausdorff dimension of the set, see~\cite{burrell2020dimensions, BurrellFF2019projection}. A similar concept of dimension interpolation between the upper box dimension and the (quasi-)Assouad dimension, called the \emph{Assouad spectrum} was initiated in~\cite{FraserYu2018AssouadSpectrum}. We refer the reader to the recent surveys~\cite{falconer2020intermedSurvey,Fraser2020Survey} for additional references in the topic of dimension interpolation. 

For $\theta<1$, a natural covering strategy to improve on the exponent given by the box dimension is to use covering sets with diameter of the two permissible extremes, i.e. either $\delta^{1/\theta}$ or $\delta$. In examples where an explicit formula is known for the intermediate dimensions, it turns out that this strategy is already optimal. This is the case for elliptical polynomial spirals~\cite{burrellFF2020SpiralsarXiv} and also for the family of countable convergent sequences~\cite{FFK2019}
\begin{equation*}
F_p=\Big\{0,\frac{1}{1^p},\frac{1}{2^p},\frac{1}{3^p},\ldots\Big\}, \text{ where } p>0.
\end{equation*}
Another large, well-known class of sets with distinct Hausdorff and box dimension are self-affine planar carpets. Already in the simplest case of Bedford--McMullen carpets, obtaining a precise formula for the intermediate dimensions seems to be a very challenging problem~\cite{FFK2019,Fraser2020Survey}. The current bounds are rather crude and far apart, in particular, the upper bound improves on the trivial bound of the box dimension only for very small values of $\theta$.

\subsection*{Main contribution}
By properly adapting the strategy of using the two extreme scales $\delta^{1/\theta}$ and $\delta$, we show that the upper intermediate dimension of a Bedford--McMullen carpet (provided it has distinct Hausdorff and box dimension) is strictly smaller than its box dimension for every $\theta<1$. This answers a question of Fraser~\cite[Question~2.1]{Fraser2020Survey}. However, in contrast to previous examples, further arguments suggest that this is not an optimal covering strategy, but rather an increasing number of scales are needed as $\theta\to 0$. Examples also indicate that the $\theta$-intermediate dimension is neither concave nor convex for the whole range of $\theta\in[0,1]$ which is also a new behaviour, see Figure~\ref{fig:DimPlotSeries}.

\subsection*{Bedford--McMullen carpets}

Independently of each other, Bedford~\cite{Bedford84_phd} and McMullen~\cite{ mcmullen84} were the first to study non-self-similar planar carpets. They split $R=[0,1]^2$ into $m$ columns of equal width and $n$ rows of equal height for some integers $n>m\geq 2$ and considered  orientation preserving maps of the form
\begin{equation*}
f_{(i,j)}(\underline{x}):= \begin{pmatrix} 1/m & 0 \\ 0 & 1/n \end{pmatrix} \begin{pmatrix} x \\ y \end{pmatrix} + \begin{pmatrix} i/m \\ j/n
\end{pmatrix}
\end{equation*}
for the index set $(i,j)\in \mathcal{A}\subseteq \{0,\ldots,m-1\}\times\{0,\ldots,n-1\}$. It is well-known that associated to the iterated function system (IFS) $\mathcal{F}=\{f_{(i,j)}\}_{(i,j)\in \mathcal{A}}$ there exists a unique non-empty compact subset $\Lambda_{\mathcal{F}}=\Lambda$ of $R$, called the attractor, such that
\begin{equation*}
\Lambda = \bigcup_{(i,j)\in \mathcal{A}} f_{(i,j)} ( \Lambda).
\end{equation*}
We call $\Lambda$ a \emph{Bedford--McMullen carpet} and refer the interested reader to the recent survey~\cite{fraser_BMcarpetSurvey_20arxiv} for further references. Figure~\ref{fig:BMPicture} shows the simplest possible example for a Beford-McMullen carpet with distinct Hausdorff and box-dimension.

\begin{figure}[H]
	\centering
	\includegraphics[width=0.85\textwidth]{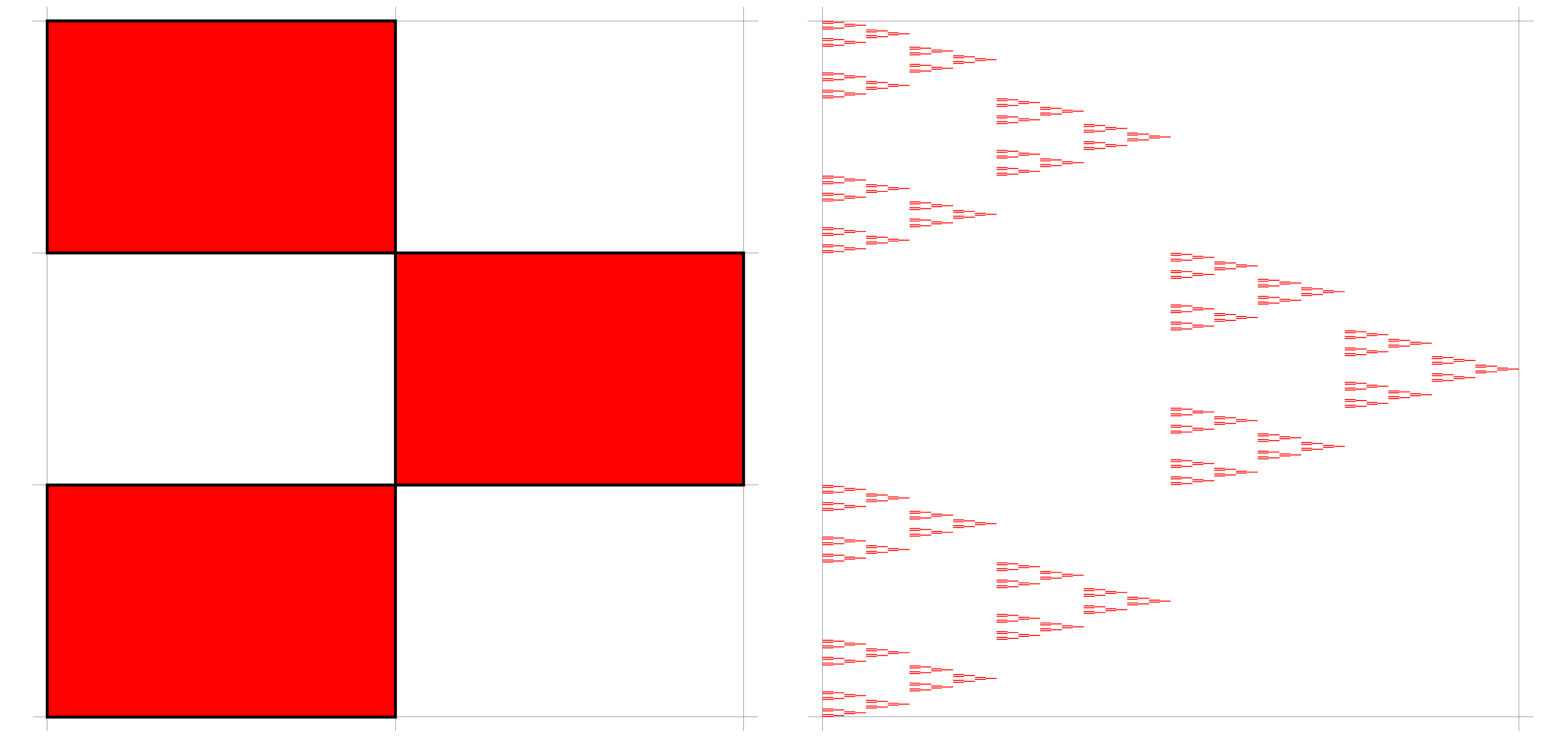}
	\caption{A Bedford-McMullen carpet with non-uniform vertical fibres. Left: the images of $[0,1]^2$ under the maps of $\mathcal{F}$. Right: the attractor $\Lambda$.}
	\label{fig:BMPicture}
\end{figure}

\subsubsection*{Notation}
Let $\Lambda$ be the Bedford--McMullen carpet associated to the IFS $\mathcal{F}$. For the remainder of the paper, we index the maps of $\mathcal{F}$ by $i\in\{1,\ldots,N\}$. We frequently use the abbreviation $[N]:=\{1,\ldots,N\}$. We can partition $[N]$ into $1<M\leq m$ sets $\mathcal{I}_1, \dots ,\mathcal{I}_M$ with cardinality $\#\mathcal{I}_{\jh}=N_{\jh}>0$ so that
\begin{equation*}
\mathcal{I}_1=\{1,\ldots,N_1\} \text{ and }\mathcal{I}_{\jh} = \{N_1+\ldots+N_{\jh-1}+1,\ldots,N_1+\ldots+N_{\jh}\} 
\end{equation*}
for $\jh=2,\ldots,M$. Moreover, this partition satisfies that
\begin{equation}\label{eq:101}
i\in\mathcal{I}_{\jh} \;\Longleftrightarrow\; f_i \text{ maps } R \text{ to the } \jh\,\text{-th non-empty column}.
\end{equation}
Formally, to keep track of this, we use the function
\begin{equation*}
\phi: \{1,2,\ldots,N\}\to\{1,2,\ldots,M\},\;\; \phi(i):=\jh\; \text{ if } i\in\mathcal{I}_{\jh}.
\end{equation*}
Throughout, $i$ is an index from $[N]$, while $\jh$ with the hat is an index corresponding to a column from $[M]:=\{1,\ldots,M\}$, see Section~\ref{sec:20} for details on symbolic notation. In Figure~\ref{fig:BMPicture} we have $N=3, M=2, \mathcal{I}_1=\{1,2\}, \mathcal{I}_2=\{3\}, \phi(1)=\phi(2)=1$ and $\phi(3)=2$. Let
\begin{equation*}
\mathbf{N}^{M}:= (N_1,\ldots, N_M) \;\text{ and }\;  \mathbf{N}:=(N_{\phi(1)},\ldots,N_{\phi(N)}) = (\overbrace{N_1,\ldots,N_1}^{N_1 \text{ times}},\ldots,\overbrace{N_M,\ldots,N_M}^{N_M \text{ times}}).
\end{equation*}
The uniform probability vector on $[N]$ and $[M]$ is denoted by
\begin{equation*}
\widetilde{\mathbf{p}} := (1/N,\ldots,1/N) \;\text{ and }\; \widetilde{\mathbf{q}}^M := (1/M,\ldots,1/M).
\end{equation*}
If we distribute $\widetilde{\mathbf{q}}^M$ uniformly within columns, we get on $[N]$ the \emph{coordinate uniform vector}
\begin{equation*}
\widetilde{\mathbf{q}}/\mathbf{N} = (1/(MN_{\phi(1)}), \ldots ,1/(MN_{\phi(N)}) ),
\end{equation*}
introduced in~\cite{FraserHowroyd_AnnAcadSciFennMath17}. The entropy of a probability vector $\mathbf{p}$ is
\begin{equation*}
H(\mathbf{p}) = -\sum_i p_i\log p_i.
\end{equation*}
In particular, $H(\widetilde{\mathbf{p}})=\log N$ and $H(\widetilde{\mathbf{q}}^M)=\log M$.
\vspace{4mm}

We say that $\Lambda$ has \emph{uniform vertical fibres} if and only if $\mathbf{N}^M= (N/M,\ldots,N/M)$, i.e. each non-empty column has the same number of maps.
Bedford and McMullen showed that the Hausdorff dimension of $\Lambda$ is equal to  
\begin{equation}\label{eq:102}
\dim_{\mathrm H}\Lambda=\frac{H(\widehat{\mathbf{p}})}{\log n} + \left( 1- \frac{\log m}{\log n}\right) \frac{H(\widehat{\mathbf{q}}^M)}{\log m},
\end{equation}
where $\widehat{\mathbf{p}}=(\widehat{p}_1,\ldots,\widehat{p}_N)$ and $\widehat{\mathbf{q}}^M=(\widehat{q}_1^M,\ldots \widehat{q}_M^M)$ are equal to
\begin{equation*}
\widehat{p}_k= N_{\ih}^{\frac{\log m}{\log n}-1} \cdot \Big(\sum_{\jh=1}^M N_{\jh}^{\frac{\log m}{\log n}}\Big)^{-1} \;\text{ and }\; \widehat{q}_{\ih}^M=N_{\ih}\cdot \widehat{p}_k \;\text{ if } k\in \mathcal{I}_{\ih}.
\end{equation*}
Similarly to $\widetilde{\mathbf{q}}/\mathbf{N}$, we can also distribute $\widehat{\mathbf{q}}^M$ evenly within columns to get a probability vector $\widehat{\mathbf{q}}/\mathbf{N}$ on $[N]$, but observe from the definition of $\widehat{\mathbf{q}}^M$ that this is simply the vector~$\widehat{\mathbf{p}}$. Bedford and McMullen also showed a similar formula for the box dimension 
\begin{equation}\label{eq:103}
\dim_{\mathrm B}\Lambda = \frac{H(\widetilde{\mathbf{p}})}{\log n} + \left(1-\frac{\log m}{\log n}\right)\frac{H(\widetilde{\mathbf{q}}^M)}{\log m}.  
\end{equation}
In particular, $\dim_{\mathrm H}\Lambda = \dim_{\mathrm B}\Lambda$ if and only if $\Lambda$ has uniform vertical fibres. Thus,
\begin{equation*}
\text{we always assume that } \Lambda \text{ has {\bf non}-uniform vertical fibres}.
\end{equation*}
We remark that \eqref{eq:102} and \eqref{eq:103} is not the usual way of writing the formula for $\dim_{\mathrm H}\Lambda$ and $\dim_{\mathrm B}\Lambda$, but it will give a rather natural way to interpolate between the two values. 

\subsection*{Results illustrated with an example}

Before any formal statements, let us compare the results of Falconer--Fraser--Kepmton~\cite{FFK2019} with our contributions. The authors of~\cite{FFK2019} proved that for any non-empty bounded set $F\subset\R^d$ the functions $\theta\mapsto\underline{\dim}_{\theta}F$ and $\theta\mapsto\overline{\dim}_{\theta}F$ are continuous for $\theta\in(0,1]$. In particular, for Bedford--McMullen carpets they gave an upper bound for $\overline{\dim}_{\theta}\Lambda$ which implies continuity also at $\theta=0$, but only improves on the trivial upper bound of $\dim_{\mathrm B}\Lambda$ for very-very small values of $\theta$. They also give a linear lower bound which shows that $\underline{\dim}_{\theta}\Lambda>\dim_{\mathrm H}\Lambda$ for every $\theta\in(0,1]$, and moreover, a general lower bound which reaches $\dim_{\mathrm B}\Lambda$ at $\theta=1$.

In contrast, our results concentrate more on larger values of $\theta$. Most notably, the new feature we show is that $\overline{\dim}_{\theta}\Lambda<\dim_{\mathrm B}\Lambda$ for every $\theta\in [0,1),$ where $\Lambda$ is a Bedford--McMullen carpet with non-uniform vertical fibres. The upper bound has a strictly positive derivative at $\theta=1$. We also give a lower bound for $\underline{\dim}_{\theta}\Lambda$ that genuinely interpolates between $\dim_{\mathrm H}\Lambda$ and $\dim_{\mathrm B}\Lambda$, which improves on the bound of~\cite{FFK2019} for $\theta$ greater than some specific $\theta^\ast>0$  depending on the carpet $\Lambda$. 

For illustration, Figure~\ref{fig:DimPlots} shows the bounds obtained in~\cite{FFK2019} and this paper in the example of Figure~\ref{fig:BMPicture}. The upper bound of~\cite{FFK2019} improves on $\dim_{\mathrm B}\Lambda$ only for $\theta\lessapprox 10^{-13}$. We remark that the (blue) plot depicting the lower bound from~\cite{FFK2019} is not precisely the one claimed there, see Section~\ref{sec:41} for an explanation. 

\begin{figure}[H]
	\centering
	\includegraphics[width=0.85\textwidth]{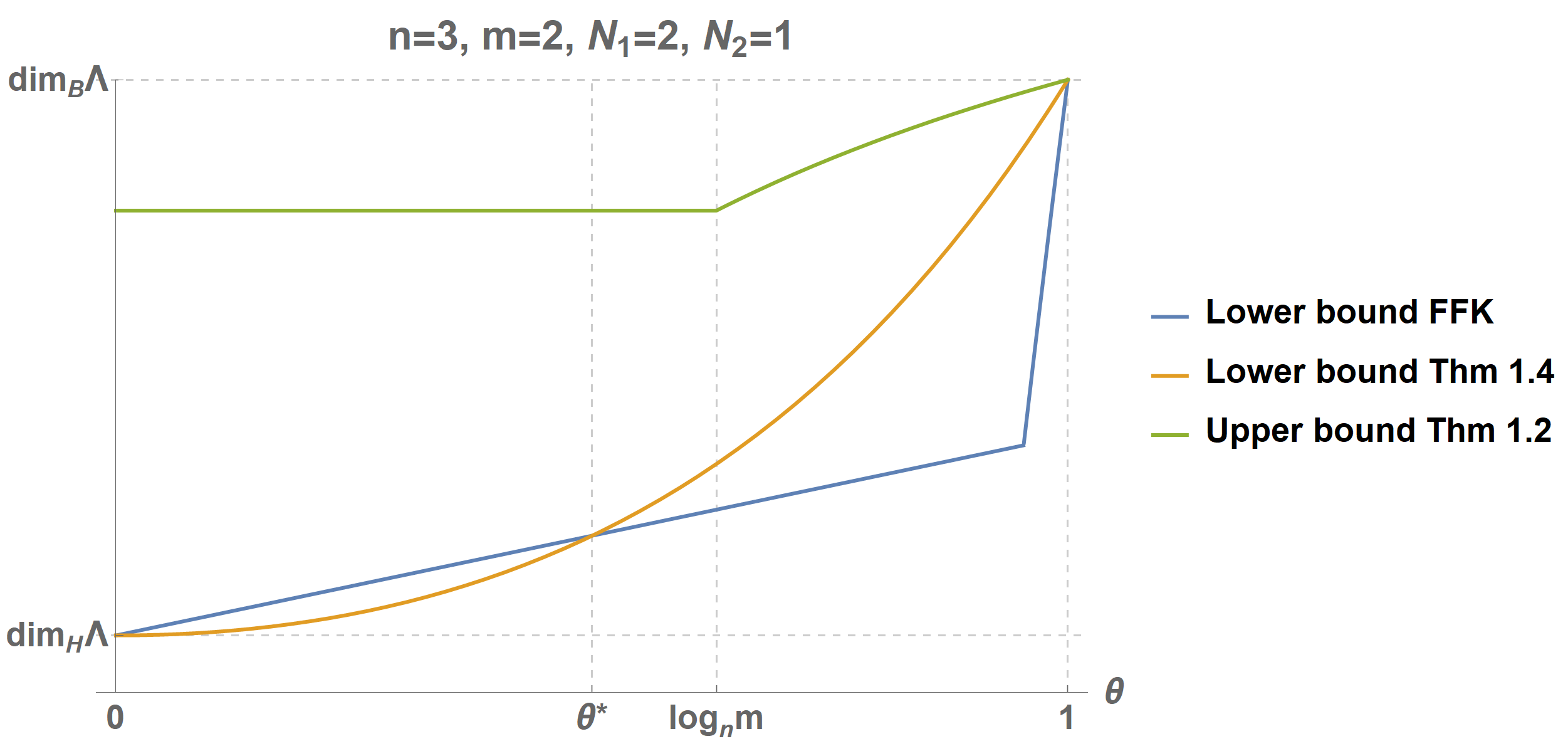}
	\caption{Bounds for the intermediate dimension of the carpet in Figure~\ref{fig:BMPicture}.}
	\label{fig:DimPlots}
\end{figure} 
\subsection{Formal statements}
Now we state our main results.

\subsubsection*{Upper bound}

The concavity of the logarithm function and non-uniform vertical fibres implies that
\begin{equation}\label{eq:107}
\overline{\log \mathbf{N}^M} := \frac{1}{M} \sum_{\jh=1}^{M} \log N_{\jh} < \log(N/M).
\end{equation}
Let $X$ denote a uniformly distributed random variable on the set $\{\log N_1,\ldots,\log N_M\}$. Then $\overline{\log \mathbf{N}^M}$ is the expected value of $X$. The large deviation rate function of $X$ is
\begin{equation}\label{eq:108}
I(x)=\sup_{\lambda\in \R} \bigg\{\lambda x - \log\bigg( \frac{1}{M} \sum_{\jh=1}^{M} N_{\jh}^{\lambda}\bigg)\bigg\}.
\end{equation} 
It is a convex function and for $x\geq \overline{\log \mathbf{N}^M}$ it is non-decreasing, it is enough to take the supremum over $\lambda\geq0$, moreover,  $I\big(\overline{\log \mathbf{N}^M}\big)=0$ \cite[Lemma 2.2.5]{DemboZeitouniLDP}.

\begin{theorem}\label{thm:00}
Let $\Lambda$ be a Bedford--McMullen carpet with non-uniform vertical fibres. Then for every $\theta\in[\log_n m,1)$ we have that
\begin{equation*}
\overline{\dim}_{\theta}\Lambda \leq \dim_{\mathrm B}\Lambda - \frac{\Delta_0(\theta)}{\log n}(1-\theta) <\dim_{\mathrm B}\Lambda,
\end{equation*}	
where $\Delta_0(\theta)\in\big(0,\log(N/M)-\overline{\log \mathbf{N}^M}\big)$ is the unique solution of 
\begin{equation}\label{eq:106}
(1-\theta)\cdot \Delta_0(\theta) = \bigg(\frac{1}{\theta}-1\bigg)\cdot I\big(\log(N/M)-\Delta_0(\theta) \big).
\end{equation}
In particular, the derivative of the upper bound remains strictly positive as $\theta\to1$.

Moreover, since $\overline{\dim}_{\theta}\Lambda$ is non-decreasing, for every $\theta\in[0,\log_n m)$
\begin{equation*}
\overline{\dim}_{\theta}\Lambda \leq \dim_{\mathrm B}\Lambda -  \frac{\Delta_0(\log_n m)}{\log n}\big(1-\log_n m\big) <\dim_{\mathrm B}\Lambda.
\end{equation*}
\end{theorem}
\begin{remark}
In explicit examples, like in Figure~\ref{fig:BMPicture}, $\Delta_0(\theta)$ can be numerically calculated.
\end{remark}
Loosely speaking, the upper bound is obtained by constructing a cover of $\Lambda$ using the two extreme scales $\delta$ and $\delta^{1/\theta}$. The cost of each part of the cover is upper bounded so that, with a properly chosen exponent $s$, it can be made arbitrarily small. Condition~\eqref{eq:106} defining $\Delta_0(\theta)$ ensures that the order of magnitude of the two parts of the cover are equal. This, however,  is \textbf{not} an optimal covering strategy, see Proposition~\ref{prop1} in Section~\ref{sec:improvements} for a cover which uses three scales in the cover and results in a better upper bound for $\overline{\dim}_{\theta}\Lambda$.

\subsubsection*{Lower bound}
For a fixed parameter $t>0$ let
\begin{equation}\label{eq:104}
\mathbf{P}^t(\theta) := \theta^t\,\widetilde{\mathbf{p}} +(1-\theta^t)\widehat{\mathbf{p}}\;\text{ and }\; \mathbf{Q}^{M,t}(\theta) := \theta^t\,\widetilde{\mathbf{q}}^M +(1-\theta^t)\widehat{\mathbf{q}}^M.
\end{equation}
Observe that $\mathbf{P}^t(\theta)$ and $\mathbf{Q}^{M,t}(\theta)$ are probability vectors for every $t>0$ and $\theta\in[0,1]$. Based on the formulas \eqref{eq:102} and \eqref{eq:103}, it is natural to introduce
\begin{equation}\label{eq:105}
\dim^t(\theta) := \frac{H(\mathbf{P}^t(\theta))}{\log n} + \left(1-\frac{\log m}{\log n}\right)\frac{H(\mathbf{Q}^{M,t}(\theta))}{\log m}. 
\end{equation}
Then, by definition, $\dim^t(0) = \dim_{\mathrm H}\Lambda$ and $\dim^t(1) = \dim_{\mathrm B}\Lambda$. Let us also define
\begin{equation*}
\psi(t,\theta):= \dim^t(\theta) - (1-\theta) \frac{H(\mathbf{P}^t(\theta))-H(\mathbf{Q}^t(\theta))}{\log n}, 
\end{equation*}
where $\mathbf{Q}^t(\theta) := \theta^t\,\widetilde{\mathbf{q}}/\mathbf{N} +(1-\theta^t)\widehat{\mathbf{p}}$. Then $\mathbf{Q}^t(\theta)$ is also a probability vector, furthermore, since $\mathbf{P}^t(0)=\mathbf{Q}^t(0)=\widehat{\mathbf{p}}$, we have that $\psi(t,0) = \dim_{\mathrm H}\Lambda$ and $\psi(t,1) = \dim_{\mathrm B}\Lambda$ for every $t$.

\begin{theorem}\label{thm:01}
Let $\Lambda$ be a Bedford--McMullen carpet with non-uniform vertical fibres. For every $\theta\in[0,1]$	
\begin{equation*}
\underline{\dim}_{\theta}\Lambda \geq \sup_{t>0}\, \psi(t,\theta). 
\end{equation*}
\end{theorem}


Loosely speaking, the proof goes by showing that a certain measure defined on $\Lambda$ using $\mathbf{P}^t(\theta)$ and $\mathbf{Q}^t(\theta)$ satisfies a variant of the mass distribution principle obtained in~\cite{FFK2019}.


\subsubsection*{Structure of paper}
Section~\ref{sec:02} introduces additional notation, defines approximate squares, outlines the covering strategy for the upper bound and the general scheme to obtain the lower bound. Section~\ref{sec:03} and \ref{sec:04} contain the proofs of Theorem~\ref{thm:00} and \ref{thm:01}, respectively. In Section~\ref{sec:improvements} we comment on how to improve on the upper bound and raise a number of questions for further research.

\section{Preliminaries}\label{sec:02}

In this section, we collect important notation and outline our strategy for proving the upper and lower bound, respectively.

\subsection{Symbolic notation}\label{sec:20}

Let $\mathcal{F}=\{f_i\}$ be an IFS generating a Bedford--McMullen carpet $\Lambda$. The map $f_i$ is indexed by $i\in\{1,2,\ldots,N\}$. Recall from~\eqref{eq:101} that we partitioned $\{1,2,\ldots,N\}$ into non-empty disjoint index sets $\mathcal{I}_{1},\ldots,\mathcal{I}_{M}$ to indicate which column $f_i$ maps to. To keep track of this, we introduced the function
\begin{equation*}
\phi: \{1,2,\ldots,N\}\to\{1,2,\ldots,M\},\;\; \phi(i):=\ih\; \text{ if } i\in\mathcal{I}_{\ih}.
\end{equation*}
For compositions of maps, we use the standard notation $f_{i_1\ldots i_n}:=f_{i_1}\circ f_{i_2}\circ\ldots \circ f_{i_n}$, where $i_{\ell}\in\{1,2,\ldots,N\}$ and so $f_{i_1\ldots i_n}$ maps to the $n$th level column indexed by the sequence $\phi(i_1)\phi(i_2)\ldots\phi(i_n).$

We define the symbolic spaces
\begin{equation*}
\Sigma=\{1,2,\ldots,N\}^\mathbb N \;\text{ and }\; \Sigma_{\mathcal{H}}=\{1,2,\ldots,M\}^\mathbb N
\end{equation*}
with elements $\ii=i_1i_2\ldots\in\Sigma$ and $\iih=\ih_1\ih_2\ldots\in\Sigma_{\mathcal{H}}$. The function $\phi$ naturally induces the map $\Phi: \Sigma\to\Sigma_{\mathcal{H}}$ defined by
\begin{equation*}
\Phi(\ii):= \iih = \phi(i_1)\phi(i_2)\ldots.
\end{equation*}
Finite words of length $n$ are either denoted with a `bar' like $\iiv=i_1\ldots i_n\in\Sigma_n$ or as a truncation $\ii|n=i_1\ldots i_n$ of an infinite word $\ii$. The length is denoted $|\cdot|$. The set of all finite length words is denoted by $\Sigma^\ast=\bigcup_n\Sigma_n$ and analogously $\Sigma_{\mathcal{H}}^\ast$. The left shift operator on $\Sigma$ and $\Sigma_{\mathcal{H}}$ is $\sigma$, i.e. $\sigma(\ii)=i_2i_3\ldots$ and $\sigma(\iih)=\ih_2\ih_3\ldots$. Slightly abusing notation, $\Phi$ is also defined on finite words: $\Phi(i_1\ldots i_n)=\phi(i_1)\ldots\phi(i_n)$.

The longest common prefix of $\ii$ and $\jj$ is denoted $\ii\wedge\jj$, i.e. its length is $|\ii\wedge\jj|=\min \{k:\; i_k\neq j_k\}-1$. This is also valid if one of them has or both have finite length. The $n$th level cylinder set of $\ii\in\Sigma$ is $[\ii|n] := \{\jj\in\Sigma:\; |\ii\wedge\jj|\geq n\}$. Similarly for $\iiv\in\Sigma_n$ and $\iih\in\Sigma_{\mathcal{H}}$. The $n$th level cylinders corresponding to $\ii$ on the attractor and $[0,1]^2$ are
\begin{equation*}
\Lambda_{n}(\ii) := f_{\ii|n}(\Lambda) \;\text{ and }\; C_{n}(\ii):=f_{\ii|n}([0,1]^2).
\end{equation*}
The sets $\{C_{n}(\ii)\}_{n=1}^\infty$ form a nested sequence of compact sets with diameter tending to zero, hence their intersection is a unique point $x\in\Lambda$. This defines the natural projection $\Pi:\Sigma\to\Lambda$
\begin{equation*}
\Pi(\ii):=\lim_{n\to\infty} \bigcap_{n=1}^\infty C_{n}(\ii)=\lim_{n\to\infty} f_{\ii|n}(\underline 0).
\end{equation*}
In particular, $\Pi([\ii|n])=\Lambda_{n}(\ii)$. The coding of a point $x\in\Lambda$ is not necessarily unique, but $\Pi$ is finite-to-one.

\subsection{Approximate squares}\label{sec:21}

The notion of an `approximate square' is crucial in the study of planar carpets. Essentially, they play the role of balls in a cover of the attractor. Since $m>n$, a cylinder set $C_K(\ii)$ has width $m^{-K}$ exponentially larger than its height~$n^{-K}$.

The correct scales at which to achieve approximately equal width and height is $K$ and 
\begin{equation*}
L(K):= \left\lfloor K \cdot \log_n m\right\rfloor<K.
\end{equation*}
In other words, $L(K)$ is the unique integer such that $m^{-K}\leq n^{-L(K)} < m^{-(K-1)}$. A level $K$ \emph{approximate square} is defined as
\begin{equation*}
B_K(\ii) := \big\{\Pi(\jj):\, |\ii\wedge\jj|\geq L(K) \text{ and } |\Phi(\ii)\wedge\Phi(\jj)|\geq K \big\}.
\end{equation*}
It is essentially a level $K$ column within a level $L(K)$ cylinder set.
\begin{remark}
One can also consider approximate squares to be the balls in the symbolic space $\Sigma$ with metric, say,
\begin{equation*}
d(\ii,\jj):= m^{-|\Phi(\ii)\wedge\Phi(\jj)|} + n^{-|\ii\wedge\jj|}.
\end{equation*}
See \cite[Section 4]{KolSimon_TriagGL2019} in a slightly more general setting.
\end{remark}

The choice of $L(K)$ implies that up to some universal multiplicative constant $C$, the diameter $|B_K(\ii)|=C m^{-K}$. We neglect $C$, since it does not influence any of the later calculations. Each approximate square can be identified with the unique sequence
\begin{equation}\label{eq:201}
B_K(\ii) = (i_1,\ldots,i_{L(K)}, \ih_{L(K)+1}, \ldots, \ih_K),
\end{equation}
where $i_1,\ldots,i_{L(K)}\in\{1,\ldots,N\}$ and $\ih_{L(K)+1}, \ldots, \ih_K\in\{1,\ldots,M\}$. Hence, denoting the set of level $K$ approximate squares by $\mathcal{B}_K$, we have that
\begin{equation*}
\# \mathcal{B}_K = N^{L(K)}\cdot M^{K-L(K)} \stackrel{\eqref{eq:103}}{=} m^{K\dim_{\mathrm B}\Lambda}.
\end{equation*}
Moreover, the number of level $K$ cylinder sets within $B_K(\ii)$ is
\begin{equation}\label{eq:204}
\# B_K(\ii) := \#\big\{ \Lambda_{K}(\jj):\, \jj|L(K) = \ii|L(K) \,\text{ and }\, \Phi(\jj|K)=\Phi(\ii|K) \big\} = \prod_{\ell=L(K)+1}^K\!\! N_{\phi(i_{\ell})}.
\end{equation}

\subsection{Covering strategy}\label{sec:22}
Up to a multiplicative constant $C$, the cost of a cover using only level $K$ approximate squares is
\begin{equation*}
\sum_{\mathcal{B}_K} |B_K(\ii)|^s = \# \mathcal{B}_K \cdot m^{-Ks} = m^{-K(s-\dim_{\mathrm B}\Lambda)},
\end{equation*}
which tends to infinity for any $s<\dim_{\mathrm B}\Lambda$. Thus, at least two scales are required to get a finite sum with an exponent $s<\dim_{\mathrm B}\Lambda$. In our covering strategy, we start from $\mathcal{B}_K$ and decide for each $B_K(\ii)\in\mathcal{B}_K$ if it is more `cost efficient' to subdivide it into smaller approximate squares or not.

When working with $\dim_{\theta}\Lambda$, we are allowed to use scales $k=K,\ldots,\lfloor K/\theta\rfloor$, corresponding to covering sets of diameter between $\delta$ and $\delta^{1/\theta}$. We first determine the number of level $k_2$ approximate squares within an approximate square of level $k_1<k_2$. Let $\mathcal{B}_{k_2}^{k_1,\ii}$ denote the set of level $k_2$ approximate squares $B_{k_2}^{k_1,\ii}(\jj)$ within the approximate square $B_{k_1}(\ii)$.
\begin{claim}\label{claim:00}
Let $K\leq k_1<k_2\leq \lfloor K/\theta\rfloor $.
\begin{enumerate}[(i)]
\item If $\theta\in[\,\log_n m,1)$, then $\#\mathcal{B}_{k_2}^{k_1,\ii} = M^{k_2-k_1}\cdot \prod_{\ell=L(k_1)+1}^{L(k_2)}N_{\ih_{\ell}}$.
\item If $\theta\in(0,\log_n m)$ and
	\begin{enumerate}
		\item $k_2\in \big(K,K\cdot\log_m n\big]$, then $\#\mathcal{B}_{k_2}^{k_1,\ii} =M^{k_2-k_1}\cdot \prod_{\ell=L(k_1)+1}^{L(k_2)}N_{\ih_{\ell}}$;
		\item $k_2\in\big (K\cdot\log_m n,K/\theta\big]$, then $\#\mathcal{B}_{k_2}^{K,\ii} = N^{L(k_2)-K}\cdot M^{k_2-L(k_2)}\cdot \prod_{\ell=L(K)+1}^{K}N_{\ih_{\ell}}$.
	\end{enumerate} 
\end{enumerate}
\end{claim}
\begin{proof}
Observe that
\begin{equation*}
\theta\in\big[\log_n m,1\big) \;\Longleftrightarrow\; L(k)\leq K \text{ for all } k=K,\ldots, \lfloor K/\theta \rfloor.
\end{equation*}
In particular, $L(k_2)\leq k_1$. Let us compare the sequences that define $B_{k_1}(\ii)$ and $B_{k_2}^{k_1,\ii}(\jj)$:
$$
\begin{array}{ccc|ccc|ccc|ccc}
i_1 & \cdots & i_{L(k_1)} & \ih_{L(k_1)+1} & \cdots & \ih_{L(k_2)} & \ih_{L(k_2)+1} & \cdots & \ih_{k_1} &&& \\
j_1 & \cdots & j_{L(k_1)} & j_{L(k_1)+1} & \cdots & j_{L(k_2)} & \jh_{L(k_2)+1} & \cdots & \jh_{k_1} & \jh_{k_1+1} & \cdots & \jh_{k_2}.
\end{array}
$$
For the first $L(k_1)$ indices $i_\ell = j_\ell$. For indices $\ell=L(k_1)+1,\ldots,L(k_2)$, we require that $\phi(j_\ell) = \ih_{\ell}$, hence the term $\prod_{\ell=L(k_1)+1}^{L(k_2)}N_{\ih_{\ell}}$. For indices $\ell=L(k_2)+1,\ldots,k_1$ there is equality again, $\ih_{\ell}=\jh_{\ell}$. Finally, there is no restriction on $\jh_{k_1+1},\ldots,\jh_{k_2}$, hence the term $M^{k_2-k_1}$.

In case $(ii/a)$ it is also true that $L(k_2)\leq k_1$. As a result, the same formula holds.

Case $(ii/b)$ can be analyzed analogously to get the formula. 
\end{proof}

We say that it is \emph{more cost efficient} to subdivide $B_{k_1}(\ii)$ into level $k_2$ approximate squares $B_{k_2}^{k_1,\ii}(\jj)$ if and only if
\begin{equation*}
m^{-k_1s}=|B_{k_1}(\ii)|^s \geq \sum_{\mathcal{B}_{k_2}^{k_1,\ii}} \big|B_{k_2}^{k_1,\ii}(\jj)\big|^s = \#\mathcal{B}_{k_2}^{k_1,\ii} \cdot m^{-k_2s}.
\end{equation*}
In particular, when $\theta\in\big[\log_n m,1\big)$, we get after algebraic manipulations that it is more cost efficient to subdivide if and only if
\begin{equation*}
s\geq \frac{\log M}{\log m} +\frac{1}{\log n} \Bigg( \frac{1}{L(k_2)-L(k_1)} \sum_{\ell=L(k_1)+1}^{L(k_2)} \log N_{\ih_{\ell}} \Bigg). 
\end{equation*}
Moreover, at the same time, we want to be able to choose $s<\dim_{\mathrm B}\Lambda = \frac{\log M}{\log m}+\frac{\log(N/M)}{\log n}$. Thus, we will subdivide if and only if
\begin{equation}\label{eq:200}
\log(N/M) > \frac{1}{L(k_2)-L(k_1)} \sum_{\ell=L(k_1)+1}^{L(k_2)} \log N_{\ih_{\ell}}.
\end{equation}
In other words, only indices $\ih_{L(k_1)+1},\ldots,\ih_{L(k_2)}$ determine whether $B_{k_1}(\ii)$ gets subdivided into level $k_2$ approximate squares or not.
\begin{remark}
At this point, one can start to appreciate the difficulty of finding an explicit formula for $\dim_{\theta}\Lambda$. Clearly, in an optimal covering strategy, different scales are present and tracking the optimal place to subdivide individual approximate squares seems hard to deal with. In the proof of Theoram~\ref{thm:00}, we use the two extreme scales $K$ and $\lfloor K/\theta\rfloor$.
\end{remark}

\subsection{General scheme for lower bound}\label{sec:23}

The aim is to define a family of probability measures supported on $\Lambda$ which satisfy the following variant of the mass distribution principle due to Falconer--Fraser--Kempton~\cite{FFK2019}.
\begin{prop}[{\cite[Proposition 2.2]{FFK2019}}]\label{prop:20}
Let $F\subset \R^2$ be a Borel set and let $0\leq \theta\leq 1$ and $s\geq 0$. Suppose that there are numbers $c,\delta_0>0$ such that for all $0<\delta\leq \delta_0$ there exists a Borel measure $\nu_{\delta}$ supported by $F$ with $\nu_{\delta}(F)>0$, and
\begin{equation*}
\nu_{\delta}(U)\leq c |U|^s \;\text{ for all Borel sets } U\subset \R^2 \text{ with } \delta^{1/\theta}\leq |U|\leq \delta.
\end{equation*} 
Then $\underline{\dim}_{\theta}F\geq s$.
\end{prop}
Fix $K>0$. To define the probability measure $\nu_K$, let us fix two probability vectors
\begin{equation*}
\mathbf{p}=(p_1,\ldots,p_N) \;\text{ and }\; \mathbf{q}=(q_1,\ldots,q_N)
\end{equation*}
which satisfy the following two properties:
\begin{align}
&H(\mathbf{p})>H(\mathbf{q}) \text{ and} \label{eq:202}\\
&\text{if } \phi(k)=\phi(\ell) \text{ for } k\neq \ell, \text{ then } p_k=p_{\ell} \text{ and } q_k=q_{\ell}. \label{eq:203}
\end{align}
The idea is to use a product measure where we use $\mathbf{p}$ until a certain level depending on $K$ and then switch to $\mathbf{q}$. The vector $\mathbf{p}$ distributes mass between the maps, while $\mathbf{q}$ really distributes mass between the columns but is re-scaled within columns evenly to get a distribution on the maps. We define $\nu_K$ on the cylinder sets $C_k(\ii)$, for $k\geq K$ to be
\begin{equation*}\label{eq:205}
\nu_K(C_k(\ii)):= \prod_{\ell=1}^{L(K)} p_{i_{\ell}} \cdot \prod_{\ell=L(K)+1}^k q_{i_{\ell}}.
\end{equation*}   
It is immediate that $\sum_{|\iiv|=k} \nu_K(C_k(\iiv)) =1$ for every $k\geq K$. Moreover, since $L(K)$ is fixed, $\nu_K$ uniquely extends to $\Lambda$ with $\nu_K(\Lambda)=1$.

Let $\mathbf{c}$ be a vector with non-negative entries and $\mathbf{p}$ be a probability vector of the same dimension as $\mathbf{c}$. For the geometric mean of $\mathbf{c}$ weighted by $\mathbf{p}$, we use the notation
\begin{equation*}
\langle\mathbf{c}\rangle_{\mathbf{p}} := \prod_i c_i^{p_i}.
\end{equation*} 
In particular, $H(\mathbf{p})=-\log \langle\mathbf{p}\rangle_{\mathbf{p}}$. Recall, $\mathbf{N}=(N_{\phi(1)},\ldots,N_{\phi(N)})$.

\begin{prop}\label{prop:21}
Let $\Lambda$ be a Bedford--McMullen carpet with non-uniform vertical fibres. Assume that $\mathbf{p}$ and $\mathbf{q}$ satisfy~\eqref{eq:202} and \eqref{eq:203}. Then for every $0\leq \theta\leq 1$
\begin{equation}\label{eq:206}
\underline{\dim}_{\theta} \Lambda \geq \frac{H(\mathbf{q})}{\log m} - \left(1-\frac{\log m}{\log n}\right)\frac{\log \langle\mathbf{N}\rangle_{\mathbf{q}}}{\log m} + \theta\, \frac{H(\mathbf{p})-H(\mathbf{q})}{\log n}\,.
\end{equation}
\end{prop}

\begin{proof}
The goal is to show that for all $K$ large enough, the measure $\nu_K$ satisfies the conditions of Proposition~\ref{prop:20} with the exponent $s$ claimed in~\eqref{eq:206}.

First observe that it is enough to consider the $\nu_K$ measure of approximate squares $B_k(\ii)$ for $k=K,\ldots, \lfloor K/\theta\rfloor$. Indeed, there is a uniform upper bound (depending only on $\Lambda$) on the number of level $k$ approximate squares any set $U$ can intersect, where $k$ is such that $m^{-k}\leq |U|< m^{-k+1}$. Moreover, property \eqref{eq:203} implies that the $\nu_K$ measure of all cylinder sets within an approximate square is the same. Thus,
\begin{equation*}
\nu_K(B_k(\ii)) = \#B_k(\ii) \cdot \nu_K(C_k(\ii)) \stackrel{\eqref{eq:204}}{=} \prod_{\ell=1}^{L(K)} p_{i_{\ell}} \cdot \prod_{\ell=L(K)+1}^k q_{i_{\ell}} \cdot \prod_{\ell=L(k)+1}^k N_{\phi(i_{\ell})}.
\end{equation*}
Taking logarithm of each side and dividing by $-k\log m$, we get
\begin{align*}
\frac{\log \nu_K(B_k(\ii))}{-k\log m} &= \frac{1}{\log n}\frac{K}{k} \bigg( \frac{-1}{L(K)}\sum_{\ell=1}^{L(K)}\log p_{i_{\ell}} \bigg) \\
&+ \frac{1}{\log m}\bigg( 1-\frac{\log m}{\log n}\frac{K}{k} \bigg) \bigg( \frac{-1}{k-L(K)}\sum_{\ell=L(K)+1}^k\log q_{i_{\ell}} \bigg) \\
&-\frac{1}{\log m} \bigg( 1-\frac{\log m}{\log n} \bigg) \bigg( \frac{1}{k-L(k)} \sum_{\ell=L(k)+1}^k \log N_{\phi(i_{\ell})} \bigg).
\end{align*}
Observe that as $K\to\infty$, all of $L(K), k-L(K)$ and $k-L(k)$ tend to~$\infty$ since $k\geq K$. Moreover, for $\nu_K$ typical $\ii$, the $p_{i_{\ell}}$ are distributed according to $\mathbf{p}$, while the $q_{i_{\ell}}$ and $N_{\phi(i_{\ell})}$ are distributed according to $\mathbf{q}$. Thus, the strong law of large numbers implies that for $\nu_K$ a.e. $\ii$ all three averages tend to their respective limits as $K\to\infty:$
\begin{align*}
\frac{-1}{L(K)}\sum_{\ell=1}^{L(K)}\log p_{i_{\ell}} &\to -\log \langle\mathbf{p}\rangle_{\mathbf{p}} = H(\mathbf{p}), \\
\frac{-1}{k-L(K)}\sum_{\ell=L(K)+1}^k\log q_{i_{\ell}} &\to -\log \langle\mathbf{q}\rangle_{\mathbf{q}} = H(\mathbf{q}), \\
\frac{1}{k-L(k)} \sum_{\ell=L(k)+1}^k \log N_{i_{\ell}} &\to \log \langle \mathbf{N}\rangle_{\mathbf{q}}.
\end{align*}
Egorov's theorem implies that for every $\varepsilon>0$ there exist an index $K_0$ and a subset $\Lambda_0\subset \Lambda$ such that for every $K\geq K_0$ we have that $\nu_K(\Lambda_0)>0$ and the three averages are simultaneously $\varepsilon$-close to their limits. Hence, for $\ii\in\Lambda_0$ and $K\geq K_0$
\begin{align*}
\frac{\log \nu_K(B_k(\ii))}{-k\log m} &\leq \frac{K}{k}\frac{H(\mathbf{p})-H(\mathbf{q})-2\varepsilon}{\log n} + \frac{H(\mathbf{q})-\varepsilon}{\log m} - \bigg( 1-\frac{\log m}{\log n} \bigg)\frac{\log \langle \mathbf{N}\rangle_{\mathbf{q}}+\varepsilon}{\log m}\ .
\end{align*}
We need a uniform upper bound for $\nu_K(B_k(\ii))$ for the full range $k=K,\ldots,\lfloor K/\theta \rfloor$. This is achieved exactly when
\begin{equation*}
\frac{K}{k}\frac{H(\mathbf{p})-H(\mathbf{q})-2\varepsilon}{\log n}\; \text{ is minimal} \;\stackrel{\eqref{eq:202}}{\Longleftrightarrow}\; k \text{ is maximal}.
\end{equation*}
Substituting back $k=\lfloor K/\theta \rfloor$, Proposition~\ref{prop:20} implies that
\begin{equation*}
\underline{\dim}_{\theta} \Lambda \geq \frac{H(\mathbf{q})-\varepsilon}{\log m} - \left(1-\frac{\log m}{\log n}\right)\frac{\log \langle\mathbf{N}\rangle_{\mathbf{q}}+\varepsilon}{\log m} + \theta\, \frac{H(\mathbf{p})-H(\mathbf{q})-2\varepsilon}{\log n}\,.
\end{equation*}
Taking limit $\varepsilon\to 0$ proves the assertion.
\end{proof}

\section{Proof of Theorem~\ref{thm:00}, the upper bound}\label{sec:03}

The proof goes by constructing a cover of $\Lambda$ using approximate squares of level $K$ and $\lfloor K/\theta\rfloor$, which correspond to covering sets of diameter $\delta$ and $\delta^{1/\theta}$. Recall from~\eqref{eq:107} 
\begin{equation*}
\overline{\log \mathbf{N}^M} = \frac{1}{M} \sum_{\ih=1}^{M} \log N_{\ih} < \log(N/M).
\end{equation*}
For the remainder of the proof we fix $\theta\in \big[\log_n m,1\big)$ and we choose
\begin{equation*}
\Delta_0=\Delta_0(\theta) \in (0,\log(N/M) - \overline{\log \mathbf{N}^M}),
\end{equation*}
which will be optimized at the end of the proof.
 
We start from the set $\mathcal{B}_K$ of level $K$ approximate squares and partition it into two sets:
\begin{equation}\label{eq:302}
\mathrm{Good}_K := \Bigg\{ B_K(\ii):\, \frac{1}{L(K/\theta)-L(K)} \sum_{\ell=L(K)+1}^{L(K/\theta)} \log N_{\ih_\ell} \,\leq\, \log(N/M)-\Delta_0 \Bigg\}
\end{equation}
and
\begin{equation*}
\mathrm{Bad}_K := \mathcal{B}_K \setminus \mathrm{Good}_K.
\end{equation*}
In light of condition~\eqref{eq:200}, it is more cost efficient to subdivide all $B_K(\ii)\in\mathrm{Good}_K$ into level $\lfloor K/\theta\rfloor$ approximate squares. Thus, let us define the cover
\begin{equation*}
\mathcal{U}_K:= \Bigg\{\mathrm{Bad}_K \;\;\cup\;\; \bigcup_{\mathrm{Good}_K}\; \bigcup_{\mathcal{B}_{K/\theta}^{K,\ii}} B_{K/\theta}^{K,\ii}(\jj) \Bigg\}.
\end{equation*}
Claim~\ref{claim:00} $(i)$ implies that the cost of this cover is
\begin{equation}\label{eq:404}
\sum_{\mathcal{U}_K} |U_i|^s = \# \mathrm{Bad}_K \cdot m^{-Ks} + \sum_{\mathrm{Good}_K} M^{K/\theta - K}\prod_{\ell=L(K)+1}^{L(K/\theta)} N_{\ih_\ell} \cdot m^{-s K/\theta}.
\end{equation}
The following lemmas guarantee that for properly chosen $s$, this cost can be made arbitrarily small for large enough $K$.

\begin{lemma}\label{lem:31}
For every $\theta\in \big[\log_n m,1\big),$ 
\begin{equation*}
\frac{\log \# \mathrm{Bad}_K}{K\log m} = \dim_{\mathrm B}\Lambda -\frac{I(\log(N/M)-\Delta_0)}{\log n}\bigg(\frac{1}{\theta}-1\bigg)+o(1), 
\end{equation*}
where $I(x)$ is the large deviation rate function of the random variable $X$ uniformly distributed on the set $\{\log N_1,\ldots,\log N_M\},$ recall~\eqref{eq:108}. As a result,
\begin{equation*}
\# \mathrm{Bad}_K \cdot m^{-Ks}\to 0, \text{ as } K\to\infty \;\;\Longleftrightarrow\;\; s> \dim_{\mathrm B}\Lambda -\frac{I(\log(N/M)-\Delta_0)}{\log n}\bigg(\frac{1}{\theta}-1\bigg).
\end{equation*} 
\end{lemma}
\begin{proof}
Recall, the fact that an approximate square $B_K(\ii)\in \mathrm{Bad}_K$ depends only on the indices $\ih_{L(K)+1},\ldots,\ih_{L(K/\theta)}$. We introduce
\begin{equation*}
\mathcal{D} := \Bigg\{(\ih_{L(K)+1},\ldots,\ih_{L(K/\theta)}):\, \frac{1}{L(K/\theta)-L(K)} \sum_{\ell=L(K)+1}^{L(K/\theta)} \log N_{\ih_\ell} \,>\, \log(N/M)-\Delta_0 \Bigg\}.
\end{equation*}
Since all other indices of $B_K(\ii)$ can be chosen freely, recall \eqref{eq:201}, we get that
\begin{equation}\label{eq:400}
\# \mathrm{Bad}_K = N^{L(K)}\cdot M^{K-L(K/\theta)}\cdot \#\mathcal{D}.
\end{equation}
Let $\{I_{\ell}\}_{\ell=L(K)+1}^{L(K/\theta)}$ be independent uniformly distributed random variables on the discrete set $\{1,\ldots,M\}$ and $X_{\ell} := \log N_{I_{\ell}}$. Then $\overline{\log \mathbf{N}^M}$ is the expected value of $X_{\ell}$. Introduce
\begin{equation*}
\overline{X}:= \frac{1}{L(K/\theta)-L(K)} \sum_{\ell=L(K)+1}^{L(K/\theta)} X_\ell.
\end{equation*}
Since all $I_\ell$ are uniformly distributed, we have that
\begin{equation*}
\mathds{P}\big(\overline{X} > \log(N/M)-\Delta_0\big) = \frac{\#\mathcal{D}}{M^{L(K/\theta)-L(K)}}.
\end{equation*}
Hence, combining this with~\eqref{eq:400}, we obtain that
\begin{equation}\label{eq:401}
\# \mathrm{Bad}_K = m^{K\dim_{\mathrm B}\Lambda} \cdot \mathds{P}\big(\overline{X} > \log(N/M)-\Delta_0\big).
\end{equation}
Cram\'er's theorem \cite[Theorem 2.1.24]{DemboZeitouniLDP} implies that for any $x>\overline{\log \mathbf{N}^M}$
\begin{equation*}
\lim_{K\to\infty}\; \frac{\log \mathds{P}\big(\overline{X} > x\big)}{L(K/\theta)-L(K)}  = -\inf_{y>x} I(y).
\end{equation*}
The infimum is equal to $I(x)$, because $I$ is continuous and non-decreasing. Applying this with $x=\log(N/M)-\Delta_0$ proves the lemma after algebraic manipulations of~\eqref{eq:401}. 
\end{proof}

\begin{lemma}\label{lem:30}
For every $\theta\in \big[\log_n m,1\big)$, if $s>\dim_{\mathrm B} \Lambda - \frac{\Delta_0}{\log n}(1-\theta)$, then
\begin{equation*}
\sum_{\mathrm{Good}_K} M^{K/\theta - K}\prod_{\ell=L(K)+1}^{L(K/\theta)} N_{\ih_\ell} \cdot m^{-s K/\theta} \to 0 \;\text{ as } K\to\infty.
\end{equation*}
\end{lemma}
\begin{proof}
For every $B_K(\ii)\in \mathrm{Good}_K$, we have the uniform upper bound 
\begin{equation*}
\prod_{\ell=L(K)+1}^{L(K/\theta)} N_{\ih_\ell} \leq \left( \frac{N}{M}\,\cdot e^{-\Delta_0} \right)^{L(K/\theta)-L(K)}.
\end{equation*}
Moreover, trivially $\# \mathrm{Good}_K \leq \#\mathcal{B}_K = N^{L(K)}M^{K-L(K)} = m^{K\dim_{\mathrm B}\Lambda}$. Thus,
\begin{align*}
\sum_{\mathrm{Good}_K} & M^{K/\theta - K}\prod_{\ell=L(K)+1}^{L(K/\theta)} N_{\ih_\ell} \cdot m^{-s K/\theta} \\
&\leq \#\mathcal{B}_K \cdot M^{K/\theta - K} \cdot \left( \frac{N}{M}\,\cdot e^{-\Delta_0} \right)^{L(K/\theta)-L(K)} \cdot m^{-s K/\theta} \\
&= m^{-K\left( (s - \dim_{\mathrm B} \Lambda) / \theta + \frac{\Delta_0}{\log n}(1/\theta-1)\right)},
\end{align*}
which tends to $0$ as $K\to\infty$ if and only if $s>\dim_{\mathrm B} \Lambda - \frac{\Delta_0}{\log n}(1-\theta)$.
\end{proof}
\begin{remark}
Lemma \ref{lem:31} shows that the bound $\# \mathrm{Good}_K<\#\mathcal{B}_K$ is essentially optimal, because $\# \mathrm{Bad}_K$ grows at an exponentially smaller rate than $\#\mathcal{B}_K$. 

The two lemmas also show that choosing $\Delta_0=0$ or $\log(N/M) - \overline{\log \mathbf{N}^M}$ would result in a bound $s>\dim_{\mathrm B}\Lambda$ for one of the parts of the cover.
\end{remark}

\begin{proof}[Proof of Theorem~\ref{thm:00}]
Fix $\theta\in \big[\log_n m,1\big)$ and let
\begin{equation*}
f_{\theta}(\Delta_0):= \Delta_0 \cdot (1-\theta) \quad\text{and}\quad g_{\theta}(\Delta_0):=I(\log(N/M)-\Delta_0)\cdot (1/\theta-1).
\end{equation*}
Lemma~\ref{lem:31} and \ref{lem:30} implies that for any $\Delta_0 \in (0,\log(N/M) - \overline{\log \mathbf{N}^M})$ if
\begin{equation}\label{eq:301}
s>\dim_{\mathrm B}\Lambda - \frac{1}{\log n}\min\left\{ f_{\theta}(\Delta_0), g_{\theta}(\Delta_0) \right\},
\end{equation}
then the cost~\eqref{eq:404} of the cover $\mathcal{U}_K$ can be made arbitrarily small.

Observe that $f_{\theta}(0)=0=g_{\theta}(\log(N/M) - \overline{\log \mathbf{N}^M})$, moreover, $f_{\theta}(\Delta_0)$ strictly increases while $g_{\theta}(\Delta_0)$ strictly decreases as $\Delta_0$ increases. Hence, there is a unique $\Delta_0(\theta)$ such that
\begin{equation}\label{eq:303}
f_{\theta}(\Delta_0(\theta)) = g_{\theta}(\Delta_0(\theta)).
\end{equation}
This is precisely condition \eqref{eq:106} in Theorem~\ref{thm:00}. It optimizes \eqref{eq:301} by making the cost of each part of the cover $\mathcal{U}_K$ to have the same order of magnitude. Since $f_{\theta}$ and $g_{\theta}$ are continuous in $\theta$, so is $\theta\mapsto\Delta_0(\theta)$. Furthermore, $\Delta_0(\theta)$ can be extended in a continuous way to be defined for $\theta=1$. Indeed, let $\theta=1-\varepsilon$, then~\eqref{eq:303} becomes
\begin{equation*}
\varepsilon\cdot \Delta_0(1-\varepsilon) = \frac{\varepsilon}{1-\varepsilon} \cdot I\big(\log(N/M)-\Delta_0(1-\varepsilon)\big).
\end{equation*} 
Hence, we define $\Delta_0(1)$ as the unique solution of $\Delta_0(1) =  I\big(\log(N/M)-\Delta_0(1)\big)$, which is clearly strictly positive. 

The conclusion of the proof goes by the definition of $\overline{\dim}_{\theta}\Lambda$, recall~\eqref{eq:100}. Fix an arbitrary $\varepsilon>0$ and $s>\dim_{\mathrm B}\Lambda-f_{\theta}(\Delta_0(\theta))/\log n$. Choose $\delta_0>0$ so small that for $K_0=K_0(\delta_0)$ defined by $m^{-K_0}\leq \delta_0<m^{-K_0+1}$, we have that
\begin{equation*}
m^{-K_0\big( s-\dim_{\mathrm B}\Lambda+f_{\theta}(\Delta_0(\theta))/\log n\big)} < \varepsilon/2.
\end{equation*}
For any $\delta<\delta_0$, we cover $\Lambda$ with $\mathcal{U}_K$, where $m^{-K}\leq \delta<m^{-K+1}$. Then $\sum_{\mathcal{U}_K}|U_i|^s<\varepsilon$ and $\delta^{1/\theta}\leq |U_i|\leq \delta$ for every $U_i\in\mathcal{U}_K$. Hence, $\overline{\dim}_{\theta}\Lambda \leq \dim_{\mathrm B} \Lambda-f_{\theta}(\Delta_0(\theta))/\log n.$ Moreover, $$\lim_{\theta\to 1}\, \frac{\mathrm{d}}{\mathrm{d} \theta} \bigg( \dim_{\mathrm B} \Lambda-\frac{f_{\theta}(\Delta_0(\theta))}{\log n}\bigg) = \frac{\Delta_0(1)}{\log n}>0.$$
\end{proof}

\section{Proof of Theorem~\ref{thm:01}, the lower bound}\label{sec:04}
	
To prove Theorem~\ref{thm:01}, we simply apply Proposition~\ref{prop:21}. Recall, it is enough to find probability vectors $\mathbf{p}=(p_1,\ldots,p_N)$ and $\mathbf{q}=(q_1,\ldots,q_N)$ which satisfy \eqref{eq:202} and \eqref{eq:203} in order to get the lower bound
\begin{equation}\label{eq:300}
\underline{\dim}_{\theta} \Lambda \geq \frac{H(\mathbf{q})}{\log m} - \left(1-\frac{\log m}{\log n}\right)\frac{\log \langle\mathbf{N}\rangle_{\mathbf{q}}}{\log m} + \theta\, \frac{H(\mathbf{p})-H(\mathbf{q})}{\log n}\,,
\end{equation}
where $\mathbf{N}=(N_{\phi(1)},\ldots,N_{\phi(N)})$. Also recall notation from~\eqref{eq:107}, \eqref{eq:104} and \eqref{eq:105}.

We claim that the choice
\begin{equation*}
\mathbf{p} = \mathbf{P}^t(\theta) = \theta^t\,\widetilde{\mathbf{p}} +(1-\theta^t)\widehat{\mathbf{p}} \;\text{ and }\; \mathbf{q} = \mathbf{Q}^t(\theta) = \theta^t\,\widetilde{\mathbf{q}}/\mathbf{N} +(1-\theta^t)\widehat{\mathbf{p}}
\end{equation*}
satisfies both \eqref{eq:202} and \eqref{eq:203}. Clearly, $\mathbf{P}^t(\theta)$ and $\mathbf{Q}^t(\theta)$ are probability vectors for every $\theta\in[0,1]$ and $t>0$. They also satisfy \eqref{eq:203}, since they are constant within each column.

\begin{claim}
$\mathbf{P}^t(\theta)$ and $\mathbf{Q}^t(\theta)$ satisfy \eqref{eq:202}, i.e. $H(\mathbf{P}^t(\theta))>H(\mathbf{Q}^t(\theta))$ for every $\theta\in(0,1]$ and $t>0$.
\end{claim}
\begin{proof}
For $\theta=0$, we have $\mathbf{P}^t(0)=\mathbf{Q}^t(0)$, thus $H(\mathbf{P}^t(0))=H(\mathbf{Q}^t(0))$. For every $\theta\in(0,1]$, $\mathbf{P}^t(\theta)\neq \mathbf{Q}^t(\theta)$, thus, $H(\mathbf{P}^t(\theta))\neq H(\mathbf{Q}^t(\theta))$. In particular,
\begin{equation}\label{eq:510}
H(\mathbf{P}^t(1)) = H(\widetilde{\mathbf{p}}) = \log N \stackrel{\eqref{eq:107}}{>} \log M + \overline{\log \mathbf{N}^M} = H(\widetilde{\mathbf{q}}/\mathbf{N}) =  H(\mathbf{Q}^t(1)).
\end{equation} 
Since $H(\mathbf{P}^t(\theta))$ and $H(\mathbf{Q}^t(\theta))$ are continuous in $\theta$ and $H(\mathbf{P}^t(\theta))$ is strictly increasing ($\widetilde{\mathbf{p}}$ is the vector with maximal entropy), the claim follows.
\end{proof}
\begin{proof}[Proof of Theorem~\ref{thm:01}]
Since $\mathbf{p}=\mathbf{P}^t(\theta)$ and $\mathbf{q}=\mathbf{Q}^t(\theta)$ satisfy \eqref{eq:202} and \eqref{eq:203}, Proposition~\ref{prop:21} implies that for every $t>0$
\begin{align*}
\underline{\dim}_{\theta} \Lambda &\geq \frac{H(\mathbf{q})}{\log m} - \bigg(1-\frac{\log m}{\log n}\bigg)\frac{\log \langle\mathbf{N}\rangle_{\mathbf{q}}}{\log m} + \theta\, \frac{H(\mathbf{p})-H(\mathbf{q})}{\log n} \pm \frac{H(\mathbf{p})}{\log n} \\
&= \frac{H(\mathbf{p})}{\log n} + \bigg(1-\frac{\log m}{\log n}\bigg) \frac{H(\mathbf{q})-\log \langle\mathbf{N}\rangle_{\mathbf{q}}}{\log m} + (\theta-1)\, \frac{H(\mathbf{p})-H(\mathbf{q})}{\log n} \\
&= \dim^t(\theta) - (1-\theta) \frac{H(\mathbf{P}^t(\theta))-H(\mathbf{Q}^t(\theta))}{\log n}=\psi(t,\theta),
\end{align*}
where we used that
\begin{align*}
H(\mathbf{Q}^t(\theta)) &= - \sum_{i=1}^N Q_i^t(\theta) \log \frac{Q_i^t(\theta) N_{\phi(i)}}{N_{\phi(i)}} = - \sum_{\ih=1}^M Q_{\ih}^{M,t}(\theta) \log Q_{\ih}^{M,t}(\theta) + \sum_{i=1}^N Q_i^t(\theta) \log N_{\phi(i)} \\
&= H(\mathbf{Q}^{M,t}(\theta)) + \log \langle \mathbf{N} \rangle_{\mathbf{Q}^t(\theta).} 
\end{align*}
Thus, $\underline{\dim}_{\theta}\Lambda \geq \sup_{t>0}\, \psi(t,\theta),$ as claimed.
\end{proof}

\begin{remark}\label{rm:40}
Another possible choice would be to take
\begin{equation*}
\mathbf{p} = \widetilde{\mathbf{p}} = (1/N,\ldots,1/N) \;\text{ and }\; \mathbf{q} = \widetilde{\mathbf{q}}/\mathbf{N} = (1/(MN_{\phi(i)}))_{i=1}^N.
\end{equation*}
They clearly satisfy both \eqref{eq:202} and \eqref{eq:203}. Thus an analogous calculation yields
\begin{equation*}
\underline{\dim}_{\theta} \Lambda \geq \dim_{\mathrm B}\Lambda - (1-\theta) \frac{\log (N/M)-\overline{\log \mathbf{N}^M}}{\log n}.
\end{equation*}
Notice that the slope of this line is equal to $\frac{\partial}{\partial \theta}\psi(t,\theta)|_{\theta=1}.$ Indeed,
\begin{equation*}
\frac{\partial}{\partial \theta}\psi(t,\theta)\Big|_{\theta=1} = \big(\dim^t\big)'(1) + \frac{H(\mathbf{P}^t(1))-H(\mathbf{Q}^t(1))}{\log n} \stackrel{\eqref{eq:510}}{=} \frac{\log (N/M)-\overline{\log \mathbf{N}^M}}{\log n}.
\end{equation*}
We leave it as an exercise to calculate that $\big(\dim^t\big)'(1)=0$.
\end{remark}

\subsection{Recovering the bound of Falconer, Fraser and Kepmton}\label{sec:41}
Observe that $\mathbf{p}=\widetilde{\mathbf{p}}$ and $\mathbf{q}=\widehat{\mathbf{p}}$ also satisfy the conditions of Proposition~\ref{prop:21}. Hence, after straightforward calculations, we get that for every $\theta\in[0,1]$
\begin{equation}\label{eq:405}
\underline{\dim}_{\theta} \Lambda \geq \dim_{\mathrm H} \Lambda + \theta\, \frac{\log N -H(\widehat{\mathbf{p}})}{\log n}.
\end{equation}
In~\cite[Proposition 4.3]{FFK2019} a better lower bound is claimed for $\theta<\log_n m$ by dividing with $\log m$ instead of $\log n$. We claim that for $n$ large enough this contradicts Theorem~\ref{thm:00}.

Indeed, fix $m,M$ and $\mathbf{N}^M$. The limit as $\theta\to\log_n m$ of the lower bound in~\cite{FFK2019} is equal to
\begin{align}
\dim_{\mathrm H} \Lambda + \frac{\log N -H(\widehat{\mathbf{p}})}{\log n} &\stackrel{\eqref{eq:102}}{=} \frac{\log N}{\log n} + \bigg(1-\frac{\log m}{\log n}\bigg)\frac{H(\widehat{\mathbf{q}}^M)\pm \log M}{\log m} \nonumber \\
&\stackrel{\eqref{eq:103}}{=} \dim_{\mathrm B}\Lambda - \bigg(1-\frac{\log m}{\log n}\bigg)\frac{\log M-H(\widehat{\mathbf{q}}^M)}{\log m}. \label{eq:403}
\end{align}
Comparing~\eqref{eq:403} with the formula in Theorem~\ref{thm:00}, we see that there is a contradiction iff
\begin{equation}\label{eq:406}
\frac{\Delta_0(\log_n m)}{\log n} > \frac{\log M-H(\widehat{\mathbf{q}}^M)}{\log m} \;\text{ for large enough } n.
\end{equation}
Using the first order Taylor expansion of $N_{\ih}^{\log_n m}$, one can calculate that $H(\widehat{\mathbf{q}}^M)=\log M+O\big( 1/\log^2 n \big)$. Thus, the right hand side in~\eqref{eq:406} is of order $O\big( 1/\log^2 n \big)$, while the left hand is $O\big( 1/\log n \big)$ (because $\inf_n\Delta_0(\log_n m)>0$). Hence, \eqref{eq:406} holds for large enough $n$.

However, upon closer inspection of the proof in~\cite{FFK2019}, in fact one only obtains the bound with $\log m$ replaced by $\log n$, as in~\eqref{eq:405}. In return, the argument is valid for all $\theta\in[0,1]$. This resolves the contradiction. In Figures~\ref{fig:DimPlots} and~\ref{fig:DimPlotSeries}, the plot attributed to~\cite{FFK2019} is in fact~\eqref{eq:405} combined with their general lower bound that converges to $\dim_{\mathrm B}\Lambda$.

We point out that the bound in~\eqref{eq:405} is the best known for small values of $\theta$. Also observe that the calculation in~\eqref{eq:403} shows that in general $\underline{\dim}_{\theta} \Lambda$ is not convex, except perhaps for the straight line connecting $\dim_{\mathrm H}\Lambda$ with $\dim_{\mathrm B}\Lambda$.



\section{Improvements and questions}\label{sec:improvements}

Many questions still remain regarding the intermediate dimensions of Bedford--McMullen carpets. The ultimate goal of finding a precise formula for $\dim_{\theta}\Lambda$ still seems out of reach (assuming that $\underline{\dim}_{\theta} \Lambda$ and $\overline{\dim}_{\theta} \Lambda$ are in fact equal), because the following argument shows that the upper bound obtained in Theorem~\ref{thm:00} is not the best possible and it suggests that an optimal covering strategy uses several different scales. 
\begin{prop}\label{prop1}
The upper bound obtained in Theorem~\ref{thm:00} is \emph{not} optimal. One can get a better bound using three levels of approximate squares. 
\end{prop}
\begin{proof}
As always in the paper, assume $\theta\in[\log_n m,1)$ and $\Lambda$ has non-uniform vertical fibres.
Let us start from the same partition of $\mathcal{B}_K$ into $\mathrm{Good}_K$ and $\mathrm{Bad}_K$ level $K$ approximate squares, recall~\eqref{eq:302}, with the choice of $\Delta_0=\Delta_0(\theta)$ from~\eqref{eq:106}. We improve on the bounds obtained in Lemmas~\ref{lem:31} and \ref{lem:30} by further subdividing $\mathrm{Good}_K$ and $\mathrm{Bad}_K$. Choose
\begin{equation*}
\Delta_1\in\big(\Delta_0,\log(N/M)-\overline{\log \mathbf{N}^M}\big) \quad\text{ and }\quad \Delta_2\in\big((1-\theta)\Delta_0,\Delta_0\big),
\end{equation*}
moreover, define
\begin{align}
\widetilde{\mathrm{Good}}_K\! &:= \!\Bigg\{ B_K(\ii)\in\mathrm{Good}_K \!: \frac{1}{L(K/\theta)-L(K)} \sum_{\ell=L(K)+1}^{L(K/\theta)}\!\!\!\! \log N_{\ih_\ell} \,\leq\, \log(N/M)-\Delta_1 \Bigg\}, \label{eq:506} \\
\widetilde{\mathrm{Bad}}_K &:= \!\Bigg\{ B_K(\ii)\in\mathrm{Bad}_K:\, \frac{1}{L(K/\theta)-L(K)} \sum_{\ell=L(K)+1}^{L(K/\theta)}\!\! \log N_{\ih_\ell} \,\geq\, \log(N/M)-\Delta_2 \Bigg\}. \nonumber
\end{align}

We still subdivide all $B_K(\ii)\in\mathrm{Good}_K$ into level $\lfloor K/\theta \rfloor$ approximate squares. On one hand, the same argument as in Lemma~\ref{lem:30} yields that the sum 
\begin{equation*}
\sum_{\widetilde{\mathrm{Good}}_K} M^{K/\theta - K}\prod_{\ell=L(K)+1}^{L(K/\theta)} N_{\ih_\ell} \cdot m^{-s K/\theta} \to 0 \;\text{ as } K\to\infty
\end{equation*}
if we choose
\begin{equation}\label{eq:500}
s>\dim_{\mathrm B} \Lambda - \frac{\Delta_1}{\log n}(1-\theta).
\end{equation}
On the other hand, to bound the sum
\begin{equation}\label{eq:501}
\sum_{\mathrm{Good}_K\setminus\widetilde{\mathrm{Good}}_K} M^{K/\theta - K}\prod_{\ell=L(K)+1}^{L(K/\theta)} N_{\ih_\ell} \cdot m^{-s K/\theta} 
\end{equation}
from above, we use that for every $B_K(\ii)\in \mathrm{Good}_K\setminus\widetilde{\mathrm{Good}}_K$
\begin{equation*}
\prod_{\ell=L(K)+1}^{L(K/\theta)} N_{\ih_\ell} \leq \left( \frac{N}{M} \right)^{L(K/\theta)-L(K)} \cdot e^{-\Delta_0(L(K/\theta)-L(K))},
\end{equation*}
and from Lemma~\ref{lem:31} we have that
\begin{equation*}
\# (\mathrm{Good}_K\setminus\widetilde{\mathrm{Good}}_K) \leq \# (\mathcal{B}_K\setminus\widetilde{\mathrm{Good}}_K) \leq m^{K\big(\dim_{\mathrm B}\Lambda -\frac{I(\log(N/M)-\Delta_1)}{\log n}(1/\theta-1)+o(1)\big)}.
\end{equation*}
Substituting these back into~\eqref{eq:501}, we get that the sum tends to $0$ as $K\to\infty$ if
\begin{equation}\label{eq:502}
s>\dim_{\mathrm B} \Lambda - \frac{\Delta_0+I(\log(N/M)-\Delta_1)}{\log n}(1-\theta).
\end{equation}

Now let's turn to the partition of $\mathrm{Bad}_K$. All $B_K(\ii)\in\widetilde{\mathrm{Bad}}_K$ we keep at level $K$. Then the argument of Lemma~\ref{lem:31} implies that 
\begin{equation}\label{eq:503}
\# \widetilde{\mathrm{Bad}}_K \cdot m^{-Ks}\to 0, \text{ as } K\to\infty \;\;\Longleftrightarrow\;\; s> \dim_{\mathrm B}\Lambda -\frac{I(\log(N/M)-\Delta_2)}{\log n}\bigg(\frac{1}{\theta}-1\bigg).
\end{equation}
For $B_K(\ii)\in(\mathrm{Bad}_K\setminus\widetilde{\mathrm{Bad}}_K)$ we have that
\begin{equation*}
\log(N/M)-\Delta_0 \;<\; \frac{1}{L(K/\theta)-L(K)} \sum_{\ell=L(K)+1}^{L(K/\theta)} \log N_{\ih_\ell} \;<\; \log(N/M)-\Delta_2.
\end{equation*}
There exists $\theta<\eta<1$ such that for every $B_K(\ii)\in(\mathrm{Bad}_K\setminus\widetilde{\mathrm{Bad}}_K)$
\begin{equation*}
\frac{1}{L(K/\eta)-L(K)} \sum_{\ell=L(K)+1}^{L(K/\eta)} \log N_{\ih_\ell} \;\leq\; \log(N/M)-\Delta_2.
\end{equation*} 
Hence, in light of condition \eqref{eq:200} it is more cost efficient to subdivide these approximate squares to level $\lfloor K/\eta \rfloor$ and the cost of this part of the cover is
\begin{align*}
\sum_{\mathrm{Bad}_K\setminus\widetilde{\mathrm{Bad}}_K} & M^{K/\eta - K}\prod_{\ell=L(K)+1}^{L(K/\eta)} N_{\ih_\ell} \cdot m^{-s K/\eta} \\
& \leq \# \mathrm{Bad}_K \cdot M^{K/\eta - K} \left( \frac{N}{M}\,\cdot e^{-\Delta_2} \right)^{L(K/\eta)-L(K)} \cdot m^{-s K/\eta} \\
& \leq C m^{K\big(\dim_{\mathrm B}\Lambda- \frac{I(\log(N/M)-\Delta_0)}{\log n}(1/\theta-1)\big)} m^{K\big( (1/\eta-1)\dim_{\mathrm B}\Lambda - \frac{\Delta_2}{\log n}(1/\eta-1) - s/\eta \big)}.
\end{align*}
This tends to $0$ as $K\to\infty$ if
\begin{equation}\label{eq:504}
s> \dim_{\mathrm B}\Lambda - \frac{\Delta_2}{\log n}(1-\eta) - \frac{I(\log(N/M)-\Delta_0)}{\log n}\bigg(\frac{1}{\theta}-1\bigg)\cdot \eta.
\end{equation}

The choice of $\Delta_0(\theta), \Delta_1$ and $\Delta_2$ implies that all the exponents in ~\eqref{eq:500}, \eqref{eq:502}, \eqref{eq:503}, and \eqref{eq:504} are strictly smaller than $\dim_{\mathrm B} \Lambda - \frac{\Delta_0(\theta)}{\log n}(1-\theta)$. As a result, we improved on the upper bound in Theorem~\ref{thm:00}.
\end{proof}

\begin{remark}\label{rm:1}
The proof also shows that the bound in Theorem~\ref{thm:00} is not the best even if we only use the two extreme scales: fix $\theta\in[\log_n m,1)$, choose $\Delta_2$ slightly smaller than $\Delta_0(\theta)$ and consider the same partition of $\mathcal{B}_K$ into $\mathrm{Good}_K$ and $\mathrm{Bad}_K$ with $\Delta_2$ instead of $\Delta_0(\theta)$. Then Lemma~\ref{lem:31} implies that we get a better a bound on $\# \mathrm{Bad}_K$ than in Theorem~\ref{thm:00}. Now choose $\Delta_1>\Delta_0(\theta)$ such that $\Delta_2+I(\log(N/M)-\Delta_1)>\Delta_0(\theta)$. Partition $\mathrm{Good}_K$ further into $\widetilde{\mathrm{Good}}_K$ and its compliment as in~\eqref{eq:506}. Then the exponents in~\eqref{eq:500} and~\eqref{eq:502} show that the cost of the $\mathrm{Good}_K$ part of the cover is also better than in Theorem~\ref{thm:00}.
\end{remark}

\subsection*{Questions}
In light of Remark~\ref{rm:1}, it is natural to ask that given $\mathrm{Good}_K$ as in~\eqref{eq:302} with $\Delta_0\in\big(0,\log(N/M)-\overline{\log \mathbf{N}^M}\big)$, what is the infimum of exponents $s$ such that
\begin{equation*}
\sum_{\mathrm{Good}_K} M^{K/\theta - K}\prod_{\ell=L(K)+1}^{L(K/\theta)} N_{\ih_\ell} \cdot m^{-s K/\theta} \to 0 \;\text{ as } K\to\infty?
\end{equation*}

We think that the upper bound is closer to the real value of $\dim_{\theta}\Lambda$ than the lower bound. So it is natural to ask, how can the argument be extended to $\theta<\log_n m$? Would it converge to $\dim_{\mathrm H}\Lambda$? Claim~\ref{claim:00} shows that the number of approximate squares within a given approximate square behaves differently for $\theta<\log_n m$, thus it is not clear what could take the place of condition~\eqref{eq:200}. Heuristically, if $\theta\in((\log_n m)^{\ell+1}, (\log_n m)^\ell)$, one could try to extend the argument to a cover in which `almost all' approximate squares are at level $\lfloor K/\theta \rfloor$ and there are some `left over' squares at levels $(\log_n m)^k$ for $k=0,1,\ldots, \ell$.  

It has already been asked whether $\dim_{\theta}\Lambda$ is strictly increasing, differentiable, or analytic~\cite[Question 2.1]{Fraser2020Survey}.
The arguments of this paper suggest a heuristic for the strictly increasing property by contradiction. Assume that there exists $\theta$ such that for every $\rho>\theta$ we have $\dim_{\theta}\Lambda<\dim_{\rho}\Lambda$, but there also exists an $\eta<\theta$ such that $\dim_{\eta}\Lambda=\dim_{\theta}\Lambda$. Take the optimal cover at $\theta$ and show that many approximate squares must be at level $\lfloor K/\theta \rfloor$. Argue that it is more cost efficient to subdivide the vast majority of these approximate squares to level $\lfloor K/\eta \rfloor$. This would improve on the exponent of this part of the cover. However, it is not clear how the exponent of the other part of the cover can be improved.

In the few examples where $\dim_{\theta}\Lambda$ is known, it is a concave function. At the end of Section~\ref{sec:41}, we saw that $\dim_{\theta}\Lambda$ is not convex in general. Further evidence also suggests that it is \emph{neither} concave in general. Figure~\ref{fig:DimPlotSeries} shows an example with $m=10=M,\,\mathbf{N}^M=\{12,2,\ldots,2\}$ and $n=12$ (left) or $n=10^5$ (right). The (green) upper bound on the left indicates that $\dim_{\theta}\Lambda$ is not concave, while as $n$ increases the (blue) lower bound of~\cite{FFK2019} approaches the straight line connecting $\dim_{\mathrm H}\Lambda$ and $\dim_{\mathrm B}\Lambda$, in-sync with~\eqref{eq:403}. Therefore, we ask, can $\dim_{\theta}\Lambda$ have phase transitions, in particular, at integer powers of $\log_n m$? Is it piecewise concave on the intervals in between phase transitions? 

\begin{figure}[H]
	\centering
	\includegraphics[width=0.98\textwidth]{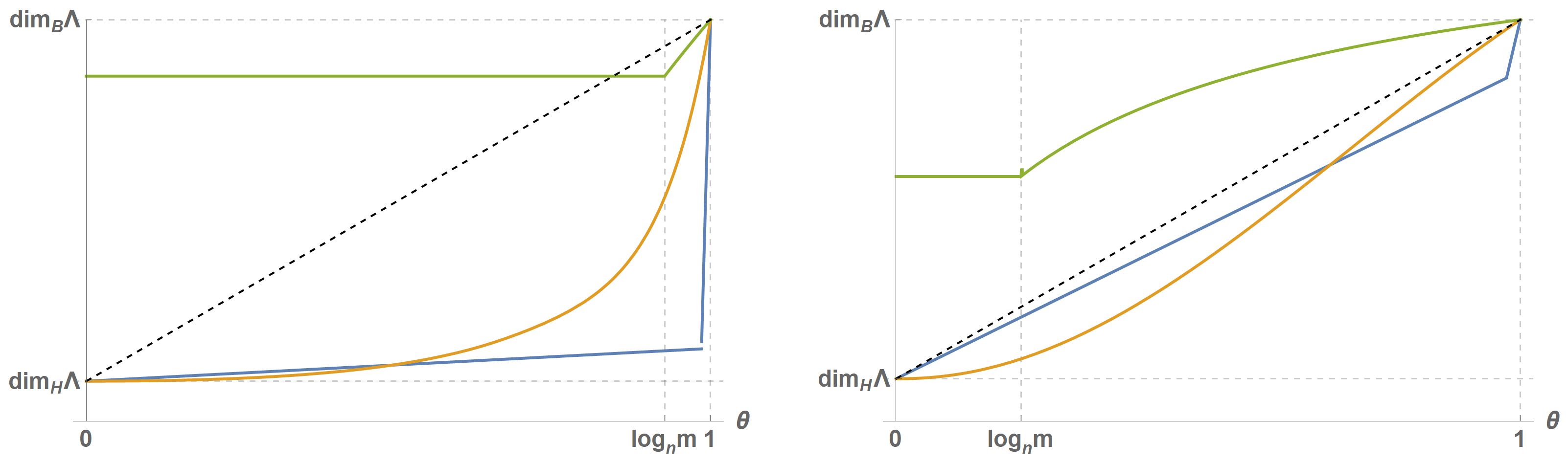}
	\caption{Example with $m=10=M,\,\mathbf{N}^M=\{12,2,\ldots,2\}$ showing that in general $\dim_{\theta}\Lambda$ is \emph{neither} concave (left, $n=12$) nor convex (right, as $n\to\infty$) for the whole range of $\theta$.}
	\label{fig:DimPlotSeries}
\end{figure} 

\subsection*{Acknowledgment}

The author was supported by a \emph{Leverhulme Trust Research Project Grant} (RPG-2019-034). The author thanks J. M. Fraser for useful discussions.
\vspace{0.2cm}

\bibliographystyle{abbrv}
\bibliography{biblio_carpet}

\end{document}